\documentclass{amsart}

\usepackage{hyperref}
\usepackage[all]{xy}
\usepackage{arydshln}

\theoremstyle{plain}
\newtheorem{theorem}{Theorem}
\newtheorem{proposition}[theorem]{Proposition}
\newtheorem{lemma}[theorem]{Lemma}
\newtheorem{corollary}[theorem]{Corollary}

\theoremstyle{definition}

\newtheorem{remark}[theorem]{Remark}

\newcommand{\f}{\varphi}

\def\AA{\mathbb A}
\newcommand{\CC}{\mathbb C}
\newcommand{\GG}{\mathbf G}
\newcommand{\PP}{\mathbb P}
\newcommand{\MM}{\mathbf M}

\newcommand{\WW}{\mathbf W}

\newcommand{\E}{\mathcal E}
\newcommand{\F}{\mathcal F}
\def\O{\mathcal O}
\newcommand{\Coker}{\mathcal Coker}

\newcommand{\Aut}{\operatorname{Aut}}
\newcommand{\diag}{\operatorname{diag}}
\newcommand{\e}{\operatorname{e}}

\newcommand{\GL}{\operatorname{GL}}
\newcommand{\PGL}{\operatorname{PGL}}
\newcommand{\h}{\operatorname{h}}
\def\H{\operatorname{H}}
\newcommand{\Hilb}{\operatorname{Hilb}}
\newcommand{\Hom}{\operatorname{Hom}}
\newcommand{\inter}{\operatorname{i}}
\newcommand{\M}{\operatorname{M}}
\newcommand{\N}{\operatorname{N}}
\newcommand{\res}{\operatorname{res}}
\newcommand{\spann}{\operatorname{span}}
\newcommand{\Sym}{\operatorname{S}}

\newcommand{\dual}{{\scriptscriptstyle \operatorname{D}}}
\newcommand{\st}{{\scriptstyle \operatorname{s}}}
\newcommand{\sst}{{\scriptstyle \operatorname{ss}}}

\newcommand{\isom}{\simeq}
\newcommand{\lra}{\longrightarrow}
\newcommand{\tensor}{\otimes}

\begin{document}

\title{On two moduli spaces of sheaves supported on quadric surfaces}

\author{Mario Maican}

\address{Center for Geometry and its Applications, Pohang University of Science and Technology,
Pohang 790-784, Republic of Korea}

\email{m-maican@wiu.edu}

\subjclass[2010]{14D20, 14D22}

\keywords{Moduli of sheaves, Semi-stable sheaves, Geometric Invariant Theory.}

\begin{abstract}
We show that the moduli space of semi-stable sheaves on a smooth quadric surface,
having dimension $1$, multiplicity $4$, Euler characteristic $2$, and first Chern class $(2, 2)$,
is the blow-up at two points of a certain hypersurface in a weighted projective space.
\end{abstract}

\maketitle

Let $\MM$ be the moduli space of Gieseker semi-stable sheaves $\F$ on $\PP^1 \times \PP^1$ having Hilbert polynomial
$P_{\F}(m) = 4m + 2$, relative to the polarization $\O(1,1)$,
and first Chern class $c_1(\F) = (2, 2)$.
Let $\M_{\PP^3}(m^2+3m+2)$ be the moduli space of Gieseker semi-stable sheaves $\F$ on $\PP^3$
having Hilbert polynomial $P_{\F}(m) = m^2 + 3m + 2$.
Such sheaves are supported on quadric surfaces.
The purpose of this note is to show that $\M_{\PP^3}(m^2+3m+2)$ is isomorphic to a certain hypersurface in a weighted projective space
(see Proposition \ref{hypersurface}) and to 
give an elementary proof of a result of Chung and Moon \cite{chung_moon}
stating that $\MM$ is the blow-up of $\M_{\PP^3}(m^2+3m+2)$ at two regular points.

Let $l$, $m$, $n$ be positive integers. Let $V$ be a vector space over $\CC$ of dimension $l$.
The reductive group $G = \big( \GL(n, \CC) \times \GL(m, \CC) \big)/\CC^*$
acts by conjugation on the vector space $\Hom(\CC^n, \CC^m \tensor V)$ of $m \times n$-matrices with entries in $V$.
The resulting good quotient
\[
\N(V; m, n) = \N(l; m, n) = \Hom(\CC^n, \CC^m \tensor V)^\sst/\!\!/ G
\]
is called a \emph{Kronecker moduli space}.
Kronecker moduli spaces arise from the study of moduli spaces of torsion-free sheaves, as in \cite{drezet_reine}.
According to \cite[Corollary 3.7]{asterisque} and \cite[Lemma 5.2]{chung_moon}, the map
\[
\Hom(2\O_{\PP^3}(-1), 2\O_{\PP^3})^\sst \lra \M_{\PP^3}(m^2 + 3m + 2), \qquad \langle \f \rangle \longmapsto \langle \Coker(\f) \rangle,
\]
is a good quotient modulo $\big( \GL(2, \CC) \times \GL(2, \CC) \big)/\CC^*$. Thus, the above moduli space is isomorphic to $\N(4; 2, 2)$.
According to \cite[Remark 3.9]{asterisque}, $\M_{\PP^3}(m^2 + 3m + 2)$ is rational;
this result was reproved in \cite{chung_moon} using the wall-crossing method.

\begin{lemma}
\label{projection}
Assume that $\N(l; m, n)$ contains stable points.
Then the same is true of $\N(k; m, n)$ for all integers $k > l$, and, moreover,
$\N(k; m, n)$ is birational to $\AA^{(k-l)mn} \times \N(l; m, n)$.
\end{lemma}

\begin{proof}
Let $U$, $V$ be vector spaces over $\CC$ of dimension $k - l$, respectively, $l$, and put $W = U \oplus V$.
The projection of $W$ onto the second factor induces a $G$-equivariant projection
\[
\pi \colon \Hom(\CC^n, \CC^m \tensor W) \lra \Hom(\CC^n, \CC^m \tensor V).
\]
From King's criterion of semi-stability \cite{king} we see that
\[
\pi^{-1} \big(\Hom(\CC^n, \CC^m \tensor V)^\st \big) \subset \Hom(\CC^n, \CC^m \tensor W)^\st.
\]
The left-hand-side, denoted by $E$, is a trivial $G$-linearized vector bundle over $\Hom(\CC^n, \CC^m \tensor V)^\st$
with fiber $\Hom(\CC^n, \CC^m \tensor U)$.
The geometric quotient map
\[
\Hom (\CC^n, \CC^m \tensor V)^\st \lra \N(V; m, n)^\st
\]
is a principal $G$-bundle, so we can apply \cite[Theorem 4.2.14]{huybrechts_lehn} to deduce that $E$ descends to a vector bundle $F$
over $\N(V; m, n)^\st$. Clearly, $F$ is the geometric quotient of $E$ by $G$, hence $F$ is isomorphic to an open subset of $\N(W; m, n)^\st$.
We conclude that $\N(W; m, n)$ is birational to $\AA^{(k-l)mn} \times \N(V; m, n)$.
\end{proof}

\begin{proposition}
\label{main_result}
\hfill
\begin{enumerate}
\item[(i)] For $l \ge 3$, $\N(l; 2, 2)$ is rational.
\item[(ii)] For $l \ge 3$ and $n \ge 1$, $\N(l; n, n+1)$ is rational.
\end{enumerate}
\end{proposition}

\begin{proof}
According to \cite[Lemma 25]{drezet_reine}, $\N(3; 2, 2)$ is isomorphic to $\PP^5$.
Identifying $\PP^5$ with the space of conic curves in $\PP^2$, the stable points correspond to irreducible conics.
Applying Lemma \ref{projection}, yields (i).

According to \cite[Propositions 4.5 and 4.6]{modules_alternatives},
the subset of $\N(3; n, n+1)$ of matrices whose maximal minors have no common factor is isomorphic
to the subset of $\Hilb_{\PP^2}(n(n+1)/2)$ of schemes that are not contained in any curve of degree $n-1$.
Thus, $\N(3; n, n+1)$ is birational to $\Hilb_{\PP^2}(n(n+1)/2)$, so it is rational.
Moreover, $\N(3; n, n+1)$ consists only of stable points. Applying Lemma \ref{projection}, yields (ii).
\end{proof}

\begin{proposition}
\label{square_rational}
For $l \ge 3$ and $n \ge 1$, $\N(l; n, n)$ is a rational variety.
\end{proposition}

\begin{proof}
The argument is inspired by \cite[Remark 3.9]{asterisque}.
In view of \cite[Section 3]{drezet_reine}, $\N(3; n, n)$ contains stable points.
This is due to the fact that we have the inequality $x < n/n < 1/x$, where $x$ is the smaller solution to the equation $x^2 - 3x +1 = 0$.
Thus, we are in the context of Lemma \ref{projection}, which asserts that $\N(l; n, n)$ is rational for $l \ge 3$ if $\N(3; n, n)$ is rational.
We may, therefore, restrict to the case when $l = 3$.
Let $V$ be a vector space over $\CC$ with basis $\{ x, y, z \}$.
An element $\f \in \Hom(\CC^n, \CC^n \tensor V)$ can be written uniquely in the form
$\f = \f_1 x + \f_2 y + \f_3 z$, where $\f_1, \f_2, \f_3 \in \Hom(\CC^n, \CC^n)$.
Let
\[
\Hom(\CC^n, \CC^n \tensor V)_0 \subset \Hom(\CC^n, \CC^n \tensor V)^\st
\]
be the open invariant subset given by the condition that $\f_1$ be invertible.
Let $X \subset \Hom(\CC^n, \CC^n \tensor V)_0$ be the closed subset given by the condition $\f_1 = I$.
The group $\PGL(n, \CC)$ acts on $X$ by conjugation.
The composite map
\[
X \hookrightarrow \Hom(\CC^n, \CC^n \tensor V)_0 \lra \Hom(\CC^n, \CC^n \tensor V)_0/G
\]
is surjective and its fibers are precisely the $\PGL(n, \CC)$-orbits.
Thus, it factors through a bijective morphism
\[
X/\PGL(n, \CC) \lra \Hom(\CC^n, \CC^n \tensor V)_0/G.
\]
In characteristic zero, bijective morphisms of irreducible varieties are birational.
We have reduced to the following problem.
Let $U$ be a complex vector space of dimension $2$ and let $\PGL(n, \CC)$ act on $\Hom(\CC^n, \CC^n \tensor U)$ by conjugation.
Then the resulting good quotient is rational.

Choose a basis $\{ y, z \}$ of $U$. An element $\psi \in \Hom(\CC^n, \CC^n \tensor U)$ can be uniquely written in the form
$\psi = y \psi_1 + z \psi_2$, where $\psi_1, \psi_2 \in \Hom(\CC^n, \CC^n)$.
Let
\[
\Hom(\CC^n, \CC^n \tensor U)_0 \subset \Hom(\CC^n, \CC^n \tensor U)
\]
be the open invariant subset given by the conditions that $\psi$ have trivial stabilizer and that $\psi_1$ be invertible and have distinct eigenvalues.
Let $Y \subset \Hom(\CC^n, \CC^n \tensor U)_0$ be the closed subset given by the condition that $\psi_1$ be a diagonal matrix.
Let $S, T \subset \PGL(n, \CC)$ be the image of the canonical embedding of the group of permutations of $n$ elements,
respectively, the subgroup of diagonal matrices.
Then $H = S T$ is a closed subgroup of $\PGL(n, \CC)$ leaving $Y$ invariant.
The composite map
\[
Y \hookrightarrow \Hom(\CC^n, \CC^n \tensor U)_0 \lra \Hom(\CC^n, \CC^n \tensor U)_0/\PGL(n, \CC)
\]
is surjective and its fibers are precisely the $H$-orbits.
Thus, it factors through a bijective morphism
\[
Y/H \longrightarrow \Hom(\CC^n, \CC^n \tensor U)_0/\PGL(n, \CC)
\]
that must be birational. We have reduced the problem to showing that $Y/H$ is rational.

Let $Y_0 \subset Y$ be the open $H$-invariant subset given by the condition that all entries of $\psi_2$ be non-zero.
Concretely, $Y_0 = D \times E$, where $D, E \subset \Hom(\CC^n, \CC^n)$ are the subset of invertible diagonal matrices with distinct entries
on the diagonal, respectively, the subset of matrices without zero entries.
The normal subgroup $T \le H$ acts trivially on $D$, hence $(D \times E)/T$ is a trivial bundle over $D$ with fiber $E/T$.
The induced action of $S = H/T$ is compatible with the bundle structure.
The stabilizer in $S$ of any $\psi_1 \in D$ acts trivially on the fiber over $\psi_1$, because it is trivial.
It follows that $(D \times E)/T$ descents to a fiber bundle $F$ over $D/S$.
Clearly, $F$ is isomorphic to $(D \times E)/H$, hence $(D \times E)/H$ is birational to $D/S \times E/T$.
Both $D/S$ and $E/T$ are rational, namely $D/S$ is isomorphic to an open subset of $\Sym^n (\AA^1) \isom \AA^n$,
while $E/T \isom (\AA^1 \setminus \{ 0 \})^{n^2-n+1}$.
In conclusion, $Y/H$ is rational.
\end{proof}

Let $r > 0$ and $\chi$ be integers. Let $\M_{\PP^2}(r, \chi)$ denote the moduli space of Gieseker semi-stable sheaves on $\PP^2$
having Hilbert polynomial $P(m) = rm+\chi$. It is well known that $\M_{\PP^2}(r, 0)$ is birational to $\N(3; r, r)$ and, if $r$ is even,
$\M_{\PP^2}(r, r/2)$ is birational to $\N(6; r/2, r/2)$. We obtain the following.

\begin{corollary}
The moduli spaces $\M_{\PP^2}(r, 0)$ and, if $r$ is even, $\M_{\PP^2}(r, r/2)$, are rational.
\end{corollary}

\noindent
The rationality of $\M_{\PP^2}(3, 0)$ and $\M_{\PP^2}(4, 2)$ is already known from \cite{lepotier_revue}.

The maps
\[
\det \colon \Hom(\CC^2, \CC^2 \tensor V) \lra \Sym^2 V, \qquad \det(\f) = \f_{11} \f_{22} - \f_{12} \f_{21},
\]
and
\[
\e \colon \Hom(\CC^2, \CC^2 \tensor V) \lra \Lambda^4 V, \qquad \e(\f) = \f_{11} \wedge \f_{22} \wedge \f_{12} \wedge \f_{21}
\]
are semi-invariant in the sense that for any $(g, h) \in G$ and $\f \in \Hom(\CC^2, \CC^2 \tensor V)$,
\[
\det((g, h) \f) = \det(g)^{-1} \det(h) \det(\f), \qquad \e((g, h) \f) = \det(g)^{-2} \det(h)^2 \e(\f).
\]
Using King's criterion of semi-stability \cite{king}, it is easy to see that $\f$ is semi-stable if and only if $\det(\f) \neq 0$ and is stable
if and only if $\det(\f)$ is irreducible in $\Sym^* V$.
In the case when $\dim(V) = 3$, the isomorphism $\N(V; 2, 2) \to \PP(\Sym^2 V)$ of \cite{drezet_reine} is given by
$\langle \f \rangle \mapsto \langle \det(\f) \rangle$.

In the sequel we will assume that $\dim(V) = 4$ and that $m = 2$, $n = 2$.
Choose bases $\{ x, y, z, w \}$ of $V$ and $\{ v_1, v_2, v_3, v_4 \}$ of $V^*$.
Consider the semi-invariant functions
\[
\epsilon, \rho \colon \Hom(\CC^2, \CC^2 \tensor V) \lra \CC, \qquad \epsilon(\f) = \inter_{v_1 \wedge v_2 \wedge v_3 \wedge v_4} \e(\f),
\]
\[
\rho(\f) = \inter_{v_1 \wedge v_2 \wedge v_3 \wedge v_4}
(\inter_{v_1} \det(\f) \wedge \inter_{v_2} \det(\f) \wedge \inter_{v_3} \det(\f) \wedge \inter_{v_4} \det(\f)).
\]
Here $\inter_v$ denotes the internal product with a vector $v \in V^*$.

\begin{proposition}
\label{functional_relation}
We have the relation $\epsilon^2 = \rho$.
\end{proposition}

\begin{proof}
Let $\{ v_1', v_2', v_3', v_4' \}$ be another basis of $V^*$ and let $\upsilon \in \GL(4, \CC)$ be the change-of-basis matrix.
With respect to this basis we define the functions $\rho'$ and $\epsilon'$ as above.
Then $\epsilon'(\f) = \det(\upsilon) \epsilon(\f)$ and $\rho'(\f) = \det(\upsilon)^2 \rho(\f)$,
hence $\epsilon(\f)^2 = \rho(\f)$ if and only if $\epsilon'(\f)^2 = \rho'(\f)$.
Put $U = \spann \{ x, y, z \}$ and let
\[
\pi \colon \Hom(\CC^2, \CC^2 \tensor V) \lra \Hom(\CC^2, \CC^2 \tensor U)
\]
be the morphism induced by the projection of $V = U \oplus \CC w$ onto the first factor.
It is enough to verify the relation on the Zariski open subset given by the condition that $\det(\pi(\f))$ be irreducible.
Changing, possibly, the basis of $U$, we may assume that $\det(\pi(\f)) = x^2 - yz$.
Since $\pi(\f)$ is stable, and since $\N(U; 2, 2)$ is isomorphic to $\PP(\Sym^2 U)$, we have
\[
\pi(\f) \sim \left[
\begin{array}{cc}
x & y \\
z & x
\end{array}
\right], \quad \text{so we may write} \quad
\f = \left[
\begin{array}{cc}
x + aw & y + bw \\
z + cw & x + dw
\end{array}
\right].
\]
We have
\[
\det(\f) = x^2 - yz + (a+d)xw - cyw - bzw + (ad-bc)w^2,
\]
\[
\e(\f) = (d-a) x \wedge y \wedge z \wedge w.
\]
Since we are free to choose the basis of $V^*$, we choose $\{ v_1, v_2, v_3, v_4 \}$ to be the dual of $\{x, y, z, w\}$.
We have
\begin{align*}
\inter_{v_1} \det(\f) & = \frac{\partial}{\partial x} \det(\f) = 2x + (a+d)w, \\
\inter_{v_2} \det(\f) & = \frac{\partial}{\partial y} \det(\f) = -z - cw, \\
\inter_{v_3} \det(\f) & = \frac{\partial}{\partial z} \det(\f) = -y - bw, \\
\inter_{v_4} \det(\f) & = \frac{\partial}{\partial w} \det(\f) = (a+d)x - cy - bz + 2(ad - bc)w,
\end{align*}
\[
\epsilon(\f) = d - a, \qquad \rho(\f) = \left|
\begin{array}{cccc}
2 & 0 & 0 & a + d \\
0 & 0 & -1 & -c \\
0 & -1 & 0 & -b \\
a + d & -c & -b & 2(ad - bc)
\end{array}
\right| = (a-d)^2.
\]
In conclusion, $\epsilon(\f)^2 = (d - a)^2 = \rho(\f)$.
\end{proof}

\noindent
Consider the action of $\CC^*$ on $\Sym^2 V \oplus \Lambda^4 V$ given by $t (q, p) = (tq, t^2 p)$ and let $\PP$
denote the weighted projective space $\big( (\Sym^2 V \oplus \Lambda^4 V) \setminus \{ 0 \} \big)/\CC^*$.
Consider the map
\[
\eta \colon \N(V; 2, 2) \lra \PP, \qquad \eta(\langle \f \rangle) = \langle \det(\f), \e(\f) \rangle.
\]
Choose coordinates on $\PP$ given by the choice of basis $\{ x, y, z, w \}$ of $V$.
In view of Proposition \ref{functional_relation}, the image of $\eta$ is contained in the hypersurface $H \subset \PP$
given by the equation $\res(q) = p^2$, where $\res(q)$ denotes the resultant of the quadratic form $q$.

\begin{proposition}
\label{hypersurface}
Assume that $\dim(V) = 4$. Then the map $\eta \colon \N(V; 2, 2) \to H$ is an isomorphism.
\end{proposition}

\begin{proof}
The singular points of the cone $\hat{H} \subset \Sym^2 V \oplus \Lambda^4 V$ over $H$ are of the form $(q, 0)$,
where $q \in \Sym^2 V$ is a singular point of the vanishing locus of the resultant.
It follows that $\hat{H}$ is regular in codimension $1$.
From Serre's criterion of normality we deduce that $H$ is normal
(condition S2 is satisfied because $\hat{H}$ is a hypersurface in a smooth variety).
Normality is inherited by a good quotient, hence $H = (\hat{H} \setminus \{ 0 \})/\CC^*$ is normal, too.
In view of the Main Theorem of Zariski, it is enough to show that $\eta$ is bijective.
Since $\N(V; 2, 2)$ is complete, and since $\N(V; 2, 2)$ and $H$ are irreducible of the same dimension,
it is enough to show that $\eta$ is injective.

Assume that $\eta(\langle \f_1 \rangle) = \eta(\langle \f_2 \rangle)$. Varying $\f_1$ and $\f_2$ in their respective orbits,
we may assume that $\det(\f_1) = \det(\f_2)$ and $\e(\f_1) = \e(\f_2)$.
If $\det(\f_1)$ is reducible, say $\det(\f_1) = u u'$ for some $u, u' \in V$, then it is easy to see that
\[
\f_1 \sim \left[
\begin{array}{cc}
u & u_1 \\
0 & u'
\end{array}
\right], \qquad \f_2 \sim \left[
\begin{array}{cc}
u & u_2 \\
0 & u'
\end{array}
\right]
\]
for some $u_1, u_2 \in V$. But then $\langle \f_1 \rangle = \langle \f_2 \rangle = \langle \diag(u, u') \rangle$.
Assume now that $\det(\f_1)$ is irreducible.
There exists a vector $w \in V$ and a subspace $U \subset V$ such that $V = U \oplus \CC w$ and $\det(\pi(\f_1))$ is irreducible
(notations as at Proposition \ref{functional_relation}).
As mentioned at Proposition \ref{functional_relation}, we may choose a basis $\{ x, y, z \}$ of $U$ such that $\det(\pi(\f_1))= x^2 - yz$, forcing
\[
\pi(\f_1) \sim \pi(\f_2) \sim \left[
\begin{array}{cc}
x & y \\
z & x
\end{array}
\right].
\]
Thus, we may write
\[
\f_1 = \left[
\begin{array}{cc}
x + a_1 w & y + b_1 w \\
z + c_1 w & x + d_1 w
\end{array}
\right], \qquad \f_2 = \left[
\begin{array}{cc}
x + a_2 w & y + b_2 w \\
z + c_2 w & x + d_2 w
\end{array}
\right].
\]
The relation $\det(\f_1) = \det(\f_2)$ yields the relations $b_1 = b_2$, $c_1 = c_2$, $a_1 + d_1 = a_2 + d_2$.
The relation $\e(\f_1) = \e(\f_2)$ yields the relation $a_1 - d_1 = a_2 - d_2$.
We conclude that $\f_1 = \f_2$, hence $\langle \f_1 \rangle = \langle \f_2 \rangle$.
\end{proof}

\begin{remark}
\label{double_cover}
It was already known to Le Potier \cite[Remark 3.8]{asterisque} that the map
\[
\det \colon \N(V; 2, 2) \lra \PP(\Sym^2 V)
\]
is a double cover branched over the locus of singular quadratic surfaces.
\end{remark}

\noindent
In the sequel, we will use the abbreviations $\O(r, s) = \O_{\PP^1 \times \PP^1}(r, s)$, $\omega = \omega_{\PP^1 \times \PP^1}$,
and $\F^\dual = {\mathcal Ext}^1_{\O}(\F, \omega)$ for a sheaf $\F$ on $\PP^1 \times \PP^1$ of dimension $1$.
We quote below \cite[Proposition 3.8]{chung_moon}.

\begin{proposition}
\label{resolutions}
The sheaves $\F$ giving points in $\MM$ are precisely the sheaves having one of the following three types of resolution:
\begin{equation}
\label{type_0}
0 \lra 2\O(-1, -1) \overset{\f}{\lra} 2\O \lra \F \lra 0,
\end{equation}
\begin{equation}
\label{type_1}
0 \lra \O(-2, -1) \lra \O(0, 1) \lra \F \lra 0,
\end{equation}
\begin{equation}
\label{type_2}
0 \lra \O(-1, -2) \lra \O(1, 0) \lra \F \lra 0.
\end{equation}
\end{proposition}

\noindent
This proposition was proved in \cite{chung_moon} by the wall-crossing method, however, it was also nearly obtained in \cite{ballico_huh}.
At \cite[Lemma 20]{ballico_huh} it is mistakenly claimed that all sheaves in $\MM$ have resolution (\ref{type_0}).
At a closer inspection, the argument of \cite[Lemma 20]{ballico_huh} shows that the sheaves in $\MM$ satisfying the conditions
$\H^0(\F^\dual(1, 0)) = 0$ and $\H^0(\F^\dual(0, 1)) = 0$ are precisely the sheaves given by resolution (\ref{type_0}).
Indeed, the exact sequence (50) in \cite{ballico_huh} reads
\begin{equation}
\label{ballico_huh_sequence}
0 \lra {\mathcal H} \lra 2\O \lra \F \lra 0,
\end{equation}
where ${\mathcal H}$ is a locally free sheaf of rank $2$ and determinant $\omega$.
Dualizing this sequence, we get the exact sequence
\begin{equation}
0 \lra 2\O(-2, -2) \lra
{\mathcal H}^\dual \isom {\mathcal H}^* \tensor \omega \isom {\mathcal H} \tensor \det({\mathcal H})^* \tensor \omega \isom {\mathcal H} \lra
\F^\dual \lra 0.
\end{equation}
From this we get the relations
\[
\h^1({\mathcal H}(1, 0)) = \h^0(\F^\dual(1, 0)) \quad \text{and} \quad \h^1({\mathcal H}(0, 1)) = \h^0(\F^\dual(0, 1)).
\]
The vanishing of $\H^1({\mathcal H}(1, 0))$ and $\H^1({\mathcal H}(0, 1))$
implies that ${\mathcal H} \isom 2\O(-1, -1)$, in which case (\ref{ballico_huh_sequence}) yields resolution (\ref{type_0}).

According to \cite[Theorem 13]{rendiconti}, if $\F$ gives a point in $\MM$, then $\F^\dual(0, 1)$ and $\F^\dual(1, 0)$ give points in the moduli space
$\MM'$ of semi-stable sheaves on $\PP^1 \times \PP^1$ having Hilbert polynomial $P(m) = 4m$ and first Chern class $c_1 = (2, 2)$.
We claim that the sheaves $\E$ giving points in $\MM'$ and satisfying the condition $\H^0(\E) \neq 0$ are precisely the structure sheaves
of curves $E \subset \PP^1 \times \PP^1$ of type $(2, 2)$.
By the argument of \cite[Lemma 9]{ballico_huh}, $\O_E$ gives a stable point in $\MM'$.
Conversely, if $\E$ gives a point in $\MM'$ and $\H^0(\E) \neq 0$, then, by the argument of \cite[Proposition 2.1.3]{dedicata},
there is an injective morphism $\O_C \to \E$ for a curve $C \subset \PP^1 \times \PP^1$.
If $C$ did not have type $(2, 2)$, then the semi-stability of $\E$ would get contradicted.
Thus, $C$ has type $(2, 2)$ and, comparing Hilbert polynomials, we see that $\O_C \isom \E$.
In conclusion, if $\H^0(\F^\dual(0, 1)) \neq 0$, then $\F \isom \O_E(0, -1)^\dual \isom \O_E(0, 1)$, hence $\F$ has resolution (\ref{type_1}).
If $\H^0(\F^\dual(1, 0)) \neq 0$, then $\F \isom \O_E(-1, 0)^\dual \isom \O_E(1, 0)$, hence $\F$ has resolution (\ref{type_2}).

We denote by $\MM_0, \MM_1, \MM_2 \subset \MM$ the subsets of sheaves given by resolution (\ref{type_0}), (\ref{type_1}), respectively, (\ref{type_2}).
Clearly, $\MM_0$ is open and $\MM_1$, $\MM_2$ are divisors isomorphic to $\PP^8$.
Let $\Hom(2\O(-1,-1), 2\O)_0$ denote the subset of injective morphisms.

\begin{corollary}
\label{generic_quotient}
The canonical map from below is a good quotient modulo $G$:
\[
\gamma \colon \Hom(2\O(-1,-1), 2\O)_0 \lra \MM_0, \qquad \gamma(\f) = \langle \Coker(\f) \rangle.
\]
\end{corollary}

\begin{proof}
According to \cite[Lemma 1]{buchdahl}, for any coherent sheaf $\F$ on $\PP^1 \times \PP^1$ there is a spectral sequence
converging to $\F$ in degree zero and to $0$ in degrees different from zero, similar to the Beilinson spectral sequence.
Its first level $\operatorname{E}_1^{ij}$ is given by
\[
\operatorname{E}_1^{ij} = 0 \quad \text{if $i > 0$ or $i < -2$},
\]
\[
\operatorname{E}_1^{0j} = \H^j(\F) \tensor \O, \quad \operatorname{E}_1^{-2,j} = \H^j(\F(-1,-1)) \tensor \O(-1, -1),
\]
and by the exact sequences
\[
\H^j(\F(0, -1)) \tensor \O(0, -1) \lra \operatorname{E}_1^{-1,j} \lra \H^j(\F(-1, 0)) \tensor \O(-1, 0).
\]
If $\F$ gives a point in $\MM_0$, then
\[
\H^0(\F) \isom \CC^2, \quad \H^1(\F) = 0, \quad \H^0(\F(-1,-1)) = 0, \quad \H^1(\F(-1, -1)) \isom \CC^2,
\]
\[
\H^0(\F(0, -1)) = 0, \quad \H^1(\F(0, -1)) = 0, \quad \H^0(\F(-1, 0)) = 0, \quad \H^1(\F(-1, 0)) = 0.
\]
Thus, $\operatorname{E}_1$ has only two non-zero terms: $\operatorname{E}_1^{-2, 1} = 2 \O(-1, -1)$ and $\operatorname{E}_1^{0,0} = 2\O$.
The relevant part of $\operatorname{E}_2$ is represented in the following table:
\[
\xymatrix
{
2 \O(-1, -1) \ar[drr]^-{\f} & 0 & 0 \\
0 & 0 & 2\O
}
\]
The sequence degenerates at $\operatorname{E}_3$, hence $\f$ is injective and $\Coker(\f) \isom \F$.
This shows that resolution (\ref{type_0}) can be obtained from the Beilinson spectral sequence of $\F$.
Arguing as at \cite[Theorem 3.1.6]{dedicata}, we can see that resolution (\ref{type_0})
can be obtained for local flat families of sheaves in $\MM_0$, hence $\gamma$ is a categorical quotient.
By the uniqueness of the categorical quotient, we deduce that $\gamma$ is a good quotient map.
\end{proof}

\noindent
We fix vector spaces $V_1$ and $V_2$ over $\CC$ of dimension $2$ and we make the identifications
\[
\PP^1 \times \PP^1 = \PP(V_1) \times \PP(V_2), \quad \H^0(\O(r, s)) = \Sym^r V_1^* \tensor \Sym^s V_2^*, \quad V = V_1^* \tensor V_2^*.
\]
Let
\[
\WW \subset \Hom \big(2\O(-1, -1) \oplus \O(-1, 0) \oplus \O(0, -1), \ \O(-1, 0) \oplus \O(0, -1) \oplus 2\O \big)
\]
be the open subset of injective morphisms $\psi$ for which $\Coker(\psi)$ is Gieseker semi-stable.
We represent $\psi$ by a matrix
\[
\psi = \left[
\begin{array}{c:c}
\psi_{11} & \psi_{12} \\
\hdashline
\psi_{21} & \psi_{22}
\end{array}
\right] = \left[
\begin{array}{cc:cc}
1 \tensor u_{12} & 1 \tensor v_{12} & a_1 & 0 \\
u_{11} \tensor 1 & v_{11} \tensor 1 & 0 & a_2 \\
\hdashline
f_{11} & f_{12} & u_{21} \tensor 1 & 1 \tensor u_{22} \\
f_{21} & f_{22} & v_{21} \tensor 1 & 1 \tensor v_{22}
\end{array}
\right],
\]
where $a_1, a_2 \in \CC$, $u_{ij}, v_{ij} \in V_j^*$, $f_{ij} \in V$.
The algebraic group
\[
\GG = \big( \Aut(2\O(-1, -1) \oplus \O(-1, 0) \oplus \O(0, -1)) \times \Aut(\O(-1, 0) \oplus \O(0, -1) \oplus 2\O) \big)/\CC^*
\]
acts on $\WW$ by conjugation. We represent elements of $\GG$ by pairs $(g, h)$, where
\[
g = \left[
\begin{array}{cc}
g_{11} & 0 \\
g_{21} & g_{22}
\end{array}
\right], \qquad h = \left[
\begin{array}{cc}
h_{11} & 0 \\
h_{21} & h_{22}
\end{array}
\right],
\]
$g_{11} \in \Aut(2\O(-1, -1))$, $h_{22} \in \Aut(2\O)$, etc.

\begin{proposition}
\label{global_quotient}
The canonical map $\theta \colon \WW \to \MM$, $\theta(\psi) = \langle \Coker(\psi) \rangle$ is a good quotient modulo $\GG$.
\end{proposition}

\begin{proof}
Let $\WW_0 \subset \WW$ be the open subset given by the condition that $\psi_{12}$ be invertible.
Concretely, $\WW_0$ is the set of morphisms $\psi$ such that $\psi_{12}$ is invertible and $\alpha(\psi) = \psi_{21} - \psi_{22} \psi_{12}^{-1} \psi_{11}$
is injective. In view of Proposition \ref{resolutions}, $\MM_0 = \theta(\WW_0)$.
The restricted map $\theta_0 \colon \WW_0 \to \MM_0$ is the composition
\[
\WW_0 \overset{\alpha}{\lra} \Hom(2\O(-1, -1), 2\O)_0 \overset{\gamma}{\lra} \MM_0,
\]
where $\gamma$ is the good quotient map from Corollary \ref{generic_quotient}.
Let $\GG_0 \unlhd \GG$ be the closed normal subgroup given by the conditions $g_{11} = c I$, $h_{22} = c I$, $c \in \CC^*$.
We have the relation $\alpha(h \psi g^{-1}) = h_{22}^{} \alpha(\psi) g_{11}^{-1}$, hence $\alpha$ is constant on the orbits of $\GG_0$.
Since any $\psi \in \WW_0$ is equivalent to
\[
\left[
\begin{array}{cc}
0 & I \\
\alpha(\psi) & 0
\end{array}
\right],
\]
it follows that the fibers of $\alpha$ are precisely the $\GG_0$-orbits, and that $\alpha$ has a section.
We deduce that $\alpha$ is a geometric quotient modulo $\GG_0$.
Since $\gamma$ is a good quotient modulo $\GG/\GG_0$, we conclude that $\theta_0$ is a good quotient modulo $\GG$.
Let $\MM^\st_0 \subset \MM_0$ be the subset of stable points.
Since $\gamma^{-1}(\MM_0^\st) \to \MM_0^\st$ is a geometric quotient modulo $\GG/\GG_0$,
we deduce that $\theta^{-1}(\MM_0^\st) \to \MM_0^\st$ is a geometric quotient modulo $\GG$.

Assume now that $\psi \in \WW \setminus \WW_0$.
Denote $\F = \Coker(\psi)$.
Then $\psi_{12} \neq 0$, otherwise $\Coker(\psi_{22})$ would be a destabilizing subsheaf of $\F$.
Thus, $\WW \setminus \WW_0$ is the disjoint union of two subsets $\WW_1$ and $\WW_2$.
The former is given by the relations $a_1 \neq 0$, $a_2 = 0$; the latter is given by the relations $a_1 = 0$, $a_2 \neq 0$.
Assume that $\psi \in \WW_1$. Then $u_{11}$, $v_{11}$ are linearly independent,
otherwise $\F$ would have a destabilizing quotient sheaf of slope zero.
Likewise, $u_{22}$, $v_{22}$ are linearly independent, otherwise $\F$ would have a destabilizing subsheaf of slope $1$.
Consider the morphism
\[
\xi \in \Hom(2\O(-1,-1) \oplus \O(0, -1), \ \O(0, -1) \oplus 2\O),
\]
\[
\xi = \left[
\begin{array}{ccc}
u_{11} \tensor 1 & v_{11} \tensor 1 & 0 \\
f_{11} - a_1^{-1} u_{21} \tensor u_{12} & f_{12} - a_1^{-1} u_{21} \tensor v_{12} & 1 \tensor u_{22} \\
f_{21} - a_1^{-1} v_{21} \tensor u_{12} & f_{22} - a_1^{-1} v_{21} \tensor v_{12} & 1 \tensor v_{22}
\end{array}
\right].
\]
Clearly, $\F \isom \Coker(\xi)$. Applying the snake lemma to the exact diagram
\[
{\small
\xymatrix@C+=28pt
{
& 0 \ar[d] & 0 \ar[d]
\\
0 \ar[r] & \O(0, -1) \ar[d] \ar[r]^-{\tiny \left[ \!\!\! \begin{array}{c} 1 \tensor u_{22} \\ 1 \tensor v_{22} \end{array} \!\!\! \right]} &
   2\O \ar[d] \ar@{->>}[rr]^-{\tiny \left[ \!\!\! \begin{array}{cc} -1 \tensor v_{22} & \!\!\! 1 \tensor u_{22} \end{array} \!\!\! \right]} & & \O(0, 1)
\\
0 \ar[r] & 2\O(-1, -1) \oplus \O(0, -1) \ar[d] \ar[r]^-{\xi} & \O(0, -1) \oplus 2\O \ar@{->>}[rr] \ar[d] & & \F
\\
\O(-2, -1) \ar@{^{(}->}[r]^-{\tiny \left[ \!\!\! \begin{array}{r} -v_{11} \tensor 1 \\ u_{11} \tensor 1 \end{array} \!\!\! \right]} &
2\O(-1, -1) \ar[d] \ar[r]^-{\tiny \left[ \!\!\! \begin{array}{cc} u_{11} \tensor 1 & \!\! v_{11} \tensor 1 \end{array} \!\!\! \right]} & \O(0, -1) \ar[rr] \ar[d] & & 0
\\
& 0 & 0
}}
\]
we obtain resolution (\ref{type_1}). This shows that $\theta(\WW_1) \subset \MM_1$. It is now easy to see that the restricted map
$\WW_1 \to \MM_1$ is surjective and that its fibers are precisely the $\GG$-orbits.
By symmetry, the same is true of the restricted map $\WW_2 \to \MM_2$.

Let $\MM^\st \subset \MM$ be the open subset of stable points and $\WW^\st = \theta^{-1}(\MM^\st)$.
We have proved above that the fibers of the restricted map $\theta^\st \colon \WW^\st \to \MM^\st$ are precisely the $\GG$-orbits.
Since $\MM^\st$ is normal (being smooth),
we can apply \cite[Theorem 4.2]{popov_vinberg} to deduce that $\theta^\st$ is a geometric quotient modulo $\GG$.
Finally, since $\MM = \MM_0 \cup \MM^\st$, we deduce that $\theta$ is a good quotient map.
\end{proof}

\noindent
Choose bases $\{ u_1, v_1 \}$ of $V_1^*$ and $\{ u_2, v_2 \}$ of $V_2^*$.
Then $x = u_1 \tensor u_2$, $y = v_1 \tensor u_2$, $z = u_1 \tensor v_2$, $w = v_1 \tensor v_2$ form a basis of $V$.
An easy calculation shows that the set of injective morphisms
\[
\Hom(2\O(-1, -1), 2\O)_0 \subset \Hom(\CC^2, \CC^2 \tensor V)
\]
is the subset of matrices whose determinant is not a multiple of $xw - yz$.
Thus, 
\[
\Hom(2\O(-1, -1), 2\O)_0/\!\!/G \isom \N(V; 2, 2) \setminus {\det}^{-1} \{ \langle xw - yz \rangle \}.
\]
According to Remark \ref{double_cover}, $\det^{-1} \{ \langle xw - yz \rangle \}$ consists of two points $\nu_1$ and $\nu_2$,
where $\epsilon(\nu_1) = 1$, $\epsilon(\nu_2) = -1$.
We saw at Corollary \ref{generic_quotient} that $\gamma$ induces an isomorphism
\[
\Hom(2\O(-1, -1), 2\O)_0/\!\!/G \lra \MM_0.
\]
The inverse of this isomorphism is denoted by
\[
\beta_0 \colon \MM_0 \lra \N(V; 2, 2) \setminus \{ \nu_1, \nu_2 \}.
\]
It is natural to ask whether $\MM$ is the blow-up of $\N(V; 2, 2)$ at $\nu_1$ and $\nu_2$.
This is, indeed, one of the main results in \cite{chung_moon}, where a blowing-down map $\beta \colon \MM \to \N(V; 2, 2)$ is constructed
via Fourier-Mukai transforms of sheaves, in view of the identification of $\N(V; 2, 2)$ with $\M_{\PP^3}(m^2 + 3m + 2)$.
We give below an alternate construction.

\begin{proposition}
\label{blowing_down}
The map $\beta_0$ extends to a blowing-down map $\beta \colon \MM \to \N(V; 2, 2)$ with exceptional divisor $\MM_1 \cup \MM_2$
and blowing-up locus $\{ \nu_1, \nu_2 \}$.
\end{proposition}

\begin{proof}
Recall that on $\MM_0 = \WW_0/\!\!/\GG$, $\beta_0$ is induced by the map sending $\psi$ to
\[
\psi_{21} - a_1^{-1} \left[
\begin{array}{c}
u_{21} \tensor 1 \\
v_{21} \tensor 1
\end{array}
\right] \left[
\begin{array}{cc}
1 \tensor u_{12} & 1 \tensor v_{12}
\end{array}
\right] - a_2^{-1} \left[
\begin{array}{c}
1 \tensor u_{22} \\
1 \tensor v_{22}
\end{array}
\right] \left[
\begin{array}{cc}
u_{11} \tensor 1 & v_{11} \tensor 1
\end{array}
\right].
\]
Equivalently, $\beta_0$ is induced by the map sending $\psi$ to
\[
a_2 \psi_{21} - a_1^{-1} a_2 \left[
\begin{array}{c}
u_{21} \tensor 1 \\
v_{21} \tensor 1
\end{array}
\right] \left[
\begin{array}{cc}
1 \tensor u_{12} & 1 \tensor v_{12}
\end{array}
\right] - \left[
\begin{array}{c}
1 \tensor u_{22} \\
1 \tensor v_{22}
\end{array}
\right] \left[
\begin{array}{cc}
u_{11} \tensor 1 & v_{11} \tensor 1
\end{array}
\right]
\]
which is defined on $\WW_0 \cup \WW_1$.
Clearly, this map factors through a morphism $\MM_0 \cup \MM_1 \to \N(V; 2, 2)$, which maps $\MM_1$ to the class of the matrix
\[
\left[
\begin{array}{c}
1 \tensor u_2 \\
1 \tensor v_2
\end{array}
\right] \left[
\begin{array}{cc}
u_1 \tensor 1 & v_1 \tensor 1
\end{array}
\right] = \left[
\begin{array}{cc}
x & y \\
z & w
\end{array}
\right],
\]
that is, to $\nu_1$.
Analogously, $\beta_0$ extends to a morphism defined on $\MM_0 \cup \MM_2$, which maps $\MM_2$ to the class of the matrix
\[
\left[
\begin{array}{c}
u_1 \tensor 1 \\
v_1 \tensor 1
\end{array}
\right] \left[
\begin{array}{cc}
1 \tensor u_2 & 1 \tensor v_2
\end{array}
\right] = \left[
\begin{array}{cc}
x & z \\
y & w
\end{array}
\right],
\]
that is, to $\nu_2$.
The two morphisms constructed thus far glue to a morphism $\beta \colon \MM \to \N(V; 2, 2)$.
Since $\nu_1$ and $\nu_2$ are smooth points, $\beta$ is a blow-down.
\end{proof}

\end{document}